\documentclass{article}
\usepackage{color}
\usepackage{bbm}
\usepackage{graphicx}
\usepackage{amsfonts,amssymb,amsmath,amsthm,mathrsfs,enumerate,multicol,esint}

\newtheorem{theorem}{Theorem}[section]

\newtheorem{lemma}[theorem]{Lemma}
\newtheorem{definition}[theorem]{Definition}

\newcommand{\ra}{\rightarrow}

\def\RR{\mathbb R}

\newcommand{\al}{\alpha}

\allowdisplaybreaks

\begin{document}
\arraycolsep=1pt

\title{Weak factorization of the Hardy space $H^p$ for small values of $p$, in the multilinear setting}

\author{Marie-Jos\'{e} S. Kuffner}

\date{}
\maketitle

\begin{abstract}
We give a weak factorization proof of the Hardy space $H^{p}(\RR^{n})$ in the multilinear setting, for $\frac{n}{n+1} < p <1$. As a consequence, we obtain a characterization of the boundedness of the commutator $[b,T]$ from $L^{r_{1}}(\mathbb R^n) \text{~x ... x~} L^{r_{m}} (\mathbb R^n) \text{ to } L^{q^\prime} (\mathbb R^n)$, where $b \in \text{Lip}_\alpha (\mathbb R^n)$, and $\frac{\alpha}{n} = \sum_{i=1}^{m} \frac{1}{r_{i}} +\frac{1}{q} - 1$.
\end{abstract}

\bigskip
\bigskip


\medskip

\section{Introduction}
\label{sec:introduction}
\setcounter{equation}{0}

\indent
It is well-known that any function $f$ in the Hardy space of the disc $H^{r}(\mathbb{D})$ can be decomposed into a product of functions in $H^{p}(\mathbb{D})$ and $H^{q}(\mathbb{D})$, where $\frac{1}{r} = \frac{1}{p}+\frac{1}{q}$. When it comes to the real variable Hardy Spaces, things become different.\

In 1976, Coifman, Rochberg and Weiss \cite{CRW} presented a weak factorization result of the Hardy space $H^{1}(\mathbb{R}^{n})$ through commutators. Their proof was based upon the duality between $H^{1}$ and BMO, a result by Fefferman and Stein \cite{FS}, and the characterization of BMO in terms of the boundedness of the commutator of a Calder\'on-Zygmund operator with multiplication operator.\

In 1981, Uchiyama \cite{U} proved a weak factorization of the Hardy space $H^{p}$ in the space of homogeneous type, for $p\leq1$. His approach allows one to obtain a weak factorization result directly, without assuming any boundedness of the commutator. In fact, the boundedness of the commutator comes as a result of weak factorization. In 2016, Chaffee \cite{Chaffee} provided a proof of the boundedness of the commutator in the multilinear setting. As a consequence of this result, one gets a weak factorization of the Hardy space $H^{p}(\mathbb{R}^{n})$. In 2017, Li and Wick \cite{LW} adopted Uchiyama's method to show weak factorization of $H^{1}(\mathbb{R}^{n})$, in the multilinear setting.\

In this paper, we extend Uchiyama's method and provide a proof of the weak factorization of $H^{p}(\mathbb{R}^{n})$, in the multilinear setting, for $\frac{n}{n+1}<p<1.$ As an application, one obtains a characterization of the boundedness of the commutator $[b, T]$ from $L^{r_{1}}(\mathbb R^n) \times \cdots \times L^{r_{m}} (\mathbb R^n) \text{ to } L^{q^\prime} (\mathbb R^n)$, where $b \in \text{Lip}_\alpha (\mathbb R^n)$, and $\frac{\alpha}{n} = \frac{1}{p} - 1$.\\

We first introduce some definitions.
\begin{definition}
A bounded tempered distribution $f$ is in the Hardy space $H^p(\mathbb{R}^{n})$ if the Poisson maximal function
\[
M(f;P) = \sup_{t>0} |(P_{t} \ast f) (x)|
\]
lies in $L^{p}(\mathbb{R}^{n})$.
\end{definition}

\begin{definition}
A function $f \in \RR^{n}$ is Lipschitz continuous of order $\alpha>0$ if there is a constant $C<\infty$ such that for all $x,y \in \RR^{n}$, we have
\[
|f(x+y) - f(x)| \leq C|y|^{\alpha}.
\]
In this case, we write $f \in Lip_{\alpha}(\RR^{n})$.\\
\end{definition}
\noindent
Note that in \cite{DRS}, the authors show that the dual of $H^{p}(\RR^{n})$ is Lip$_{\alpha}(\RR^{n})$; a key fact that will be used later on in this paper. \\

 \noindent
 We now review the notion of multilinear Calder\'on-Zygmund theory studied in \cite{GT}.
 \begin{definition}
 Let $0<\epsilon, A < \infty$. A locally integrable function $K(y_{0},y_{1}, \cdots, y_{m})$ defined away from the diagonal $\{ y_{0}=y_{1} = \cdots = y_{m} \}$ in $(\mathbb{R}^{n})^{m+1}$ is said to be an $m$-linear Calder\'on-Zygmund kernel with constants $\epsilon, A$ if
 \begin{itemize}
 \item $K$ satisfies a size condition:
 \[
 |K(y_{0}, y_{1}, \cdots, y_{m})| \leq \frac{A}{(\sum_{k,l=0}^{m} |y_{k}-y_{l}|)^{mn}},
 \]
 and
 \item $K$ satisfies a smoothness condition:
 \[
 |K(y_{0}, \cdots, y_{j}, \cdots, y_{m}) - K(y_{0}, \cdots, y'_{j}, \cdots, y_{m})| \leq \frac{A|y_{j}-y_{j'}|^{\epsilon}}{(\sum_{k,l=0}^{m} |y_{k}-y_{l}|)^{mn+\epsilon}},
 \]
 whenever $0\leq j \leq m$ and $ |y_{j}-y_{j'}| \leq \frac{1}{2} \max_{0\leq k\leq m}|y_{j}-y_{k}|$.
 \end{itemize}
\end{definition}

\begin{definition}
Let $0<\epsilon, A < \infty$. An $m$-linear operator $T: L^{r_{1}}(\mathbb R^n) \times \cdots \times L^{r_{m}} (\mathbb R^n) \text{ to } L^{p} (\mathbb R^n)$ is said to be a Calder\'on-Zygmund operator if $T$ is associated with the $m$-linear Calder\'on-Zygmund kernel $K$, i.e.,
\[
T(f_{1}, \cdots, f_{m})(x) = \int_{\mathbb{R}^{mn}} K(x,y_{1}, \cdots, y_{m}) \Pi_{j=1}^{m} f_{j}(y_{j}) dy_{1}\cdots dy_{m},
\]
for all $x \notin \cap_{j=1}^{m} \text{ supp}(f_{j})$, where $f_{1}, \cdots, f_{m}$ are $m$ functions on $\mathbb{R}^{n}$ with $\cap_{j=1}^{m} \text{ supp}(f_{j}) \neq \phi$,
and 
\[
\frac{1}{p} = \sum_{i=1}^{m} \frac{1}{r_{i}}.
\]
The $lth$ partial adjoint of $T$ is 
\[
T_{l}^{*} (f_{1}, \cdots, f_{m})(x) = \int_{\mathbb{R}^{mn}} K(y_{l}, y_{1}, \cdots, y_{l-1}, x, y_{l+1}, \cdots, y_{m}) \Pi_{j=1}^{m} f_{j}(y_{j}) dy_{1} \cdots dy_{m}.
\]
\end{definition}

Note that Calder\'on-Zygmund operators are originally defined on Schwartz function spaces $\mathcal{S}(\mathbb{R}^{n})$. In \cite{GT}, the authors show that a Calder\'on-Zygmund operator $T$ indeed extends to a bounded operator from $L^{r_{1}}(\mathbb R^n) \text{~x ... x~} L^{r_{m}} (\mathbb R^n)$ to $L^{p}(\mathbb{R}^{n})$, provided $\frac{1}{p}= \sum_{i=1}^{m} \frac{1}{r_{i}}$.

\begin{definition}
We say that a Calder\'on-Zygmund operator $T$ is $mn$-homogeneous if the kernel $K$ of $T$ satisfies
\[
|K(x_{0}, \cdots, x_{m})| \geq \frac{C}{N^{mn}},
\]
for $m+1$ pairwise disjoint balls $B_{0}(x_{0},r), \cdots, B_{m}(x_{m},r)$ satisfying $|x_{0}-x_{l}| \simeq Nr$ for all $x_{l}, l=1, \cdots, m$, where $r>0$ and $N$ a large number.
\end{definition}

\section{Statement of Main Results}
\begin{theorem}\label{Thm1Multi}
Let $T$ be an $m$-linear Calder\'on-Zygmund operator that is $mn$-homogeneous in the lth component, with $\frac{n}{n+1} < \epsilon < 1$, where $1\leq l \leq m$. Then, for every $f \in H^{p}(\RR^{n})$, there exist sequences $ \{ \lambda_{j}^{k} \} \subseteq l^{p}, \{ g_{j}^{k}\} \subseteq L^{q}(\RR^{n}), \{h_{j,1}^{k}\} \subseteq L^{r_{1}}(\RR^{n}), \cdots, \{h_{j,m}^{k}\} \subseteq L^{r_{m}}(\RR^{n})$, with $\frac{1}{q} + \frac{1}{r_{1}} + \cdots + \frac{1}{r_{m}} = \frac{1}{p}$, such that
\begin{align}\label{factorization}
f=\sum_{k=1}^{\infty} \sum_{j=1}^{\infty} \lambda_{j}^{k} \Pi_{l}(g_{j}^{k}, h_{j,1}^{k}, \cdots, h_{j,m}^{k}) \text{ in } H^{p}(\RR^{n}),
\end{align}
where $$\Pi_{l}(g_{j}^{k}, h_{j,1}^{k}, \cdots, h_{j,m}^{k}) = h_{j,l}^{k} T_{l}^{*} (h_{j,1}^{k}, \cdots, h_{j,l-1}^{k},g_{j}^{k},h_{j,l+1}^{k}, \cdots, h_{j,m}^{k}) - g_{j}^{k} T(h_{j,1}^{k}, \cdots,  h_{j,m}^{k}).$$
Moreover, we have
\[
\|f\|_{H^{p}(\RR^{n})} \approx \inf \big( \{ \sum_{k=1}^{\infty} \sum_{j=1}^{\infty} |\lambda_{j}^{k}|^{p} \|g_{j}^{k}\|_{L^{q}(\RR^{n})} \|h_{j,1}^{k}\|_{L^{r_{1}}(\RR^{n})} \cdots \|h_{j,m}^{k}\|_{L^{r_{m}}(\RR^{n})}\} \big)^{1/p},\]
 where the infimum is taken over all possible representations of $f$ that satisfy  \eqref{factorization}.

\end{theorem}

As a consequence of this Theorem, we obtain a new characterization of $\text{Lip}_\alpha (\mathbb R^n)$ in terms of the boundedness of the commutators with the multilinear Riesz transforms, where the commutator operator (of some function $b$ against an operator $T$) is defined by 
\[
[b,T]_{l}(f_{1}, \cdots, f_{m}) = T(f_{1}, \cdots, bf_{l}, \cdots, f_{m}) - bT(f_{1}, \cdots, f_{m}).
\]
Notice that the multiplication operator $\Pi_{l}$ defined in Theorem \ref{Thm1Multi} is the dual of the commutator operator $[b,T]_{l}$.

\begin{theorem}\label{Thm2Multi}
Let $T$ be an $m$-linear Calder\'on-Zygmund operator, that is $mn$-homogeneous in the lth component, with $\frac{n}{n+1} < \epsilon < 1$, $b \in L^{1}_{loc}(\RR^{n})$ and $\alpha>0$ such that $\frac{\alpha}{n} = \frac{1}{p} - 1$. Suppose that $b \in Lip_{\alpha}(\RR^{n})$, then
\[
\|b\|_{Lip_{\alpha}(\RR^{n})} \approx \| [b,T]_{l} \|_{L^{r_{1}}(\RR^{n}) \times \cdots \times L^{r_{m}}(\RR^{n}) \ra L^{q^{\prime}}(\RR^{n})},
\]
where $\frac{1}{q} + \frac{1}{r_{1}} + \cdots + \frac{1}{r_{m}} = \frac{1}{p}$ and $q'$ is the dual exponent of $q$.

\end{theorem}

\section{Preliminaries}
\setcounter{equation}{0}
\label{sec:preliminaries}

In this section, we recall some definitions and theorems we need in order to prove our main results.

\begin{definition}\label{def of atom}
A function $a$ is called an $L^\infty$-atom for $H^p(\mathbb{R}^{n})$ (or simply an $H^{p}(\RR^{n})$-atom) if there exists a ball $B \subset \mathbb{R}^{n}$ such that supp ${a} \subset B$, $\|{a}\|_{\infty} \leq |B|^{-1/p}$ and $\int_{0}^{\infty} x^\gamma a(x) dx = 0$ for all multi-indices $\gamma$ with $|\gamma| \leq [\frac{n}{p} - {n}]$.
\end{definition}

\noindent
Notice that since $p>\frac{n}{n+1}$ in our case, then we have $\gamma = 0$.
\\

\begin{lemma}
Let $f$ be a function on $\mathbb{R}^{n}$ satisfying:
\begin{itemize}
\item $\int_{\mathbb{R}^{n}} f(x)dx = 0$\\
\item there exist $y_{1}, y_{2} \in \RR^{n}, r\in\RR$ and $N$ large such that
\[
|f(x)| \lesssim C_{1}\chi_{B(y_{1},r)}(x) + C_{2}\chi_{B(y_{2},r)}(x) \text{ for } |y_{1} - y_{2}|=Nr.
\]
\end{itemize}
Then,
\[
 f\in H^{p}(\RR^{n}) \text{ and } \|f\|_{H^{p}(\RR^{n})} \lesssim N^{n(1-p)} \log_{2} N (C_{1} |B(y_{1},r)| + C_{2} |B(y_{2}, r)|)
\]
\end{lemma}

\begin{proof}
We will proceed with the proof of this lemma using the atomic decomposition characterization of $H^{p}({\RR^{n}})$. Since $|f(x)| \lesssim C_{1}\chi_{B(y_{1},r)}(x) + C_{2}\chi_{B(y_{2},r)}(x)$, we write
\[
f=f_{1}+f_{2} \text{ where supp } f_{i} \subseteq B(y_{i},r) \text{ for } i=1,2.
\]
Let $J_{0}$ be the smallest integer larger than $\log_{2} \frac{|y_{1}-y_{2}|}{r}$ and for $k=1, \cdots, J_{0}$, let
\[
\alpha_{i}^{k} = \frac{|B(y_{i},r)|}{|B(y_{i},2^{k}r)|} \langle f_{i} \rangle_{B(y_{i},r)},
\]
and
\[
f_{i}^{k} = \alpha_{i}^{k-1} \chi_{B(y_{i},2^{k-1}r)} - \alpha_{i}^{k} \chi_{B(y_{i},2^{k}r)},
\]
where $\alpha_{i}^{0} = f_{i}$ and  $\langle f_{i} \rangle_{B(y_{i},r)} = \frac{1}{|B(y_{i},r)|} \int_{B(y_{i},r)} f_{i}(x) dx$, the average of $f_{i}$ on $B(y_{i},r)$. 
Then,
\begin{align*}
f &= f_{1}+f_{2}\\
&=f_{1} + f_{2} - \sum_{i=1}^{2} \al_{i}^{1} \chi_{B(y_{i},2r)} + \sum_{i=1}^{2} \al_{i}^{1} \chi_{B(y_{i},2r)}\\
&= \sum_{i=1}^{2} (f_{i} -  \al_{i}^{1} \chi_{B(y_{i},2r)}) + \sum_{i=1}^{2} \al_{i}^{1} \chi_{B(y_{i},2r)}\\
&= \sum_{i=1}^{2} f_{i}^{1} + \sum_{i=1}^{2} \al_{i}^{1} \chi_{B(y_{i},2r)}\\
&= \sum_{i=1}^{2} f_{i}^{1} + \sum_{i=1}^{2} \al_{i}^{1} \chi_{B(y_{i},2r)} -\al_{i}^{2} \chi_{B(y_{i},2^{2}r)} + \al_{i}^{2} \chi_{B(y_{i},2^{2}r)} \\
&= \sum_{i=1}^{2} f_{i}^{1} + \sum_{i=1}^{2} f_{i}^{2} + \sum_{i=1}^{2} \alpha_{i}^{2} \chi_{B(y_{i},2^{2}r)} - \alpha_{i}^{3} \chi_{B(y_{i},2^{3}r)} + \alpha_{i}^{3} \chi_{B(y_{i},2^{3}r)}\\
&= \sum_{i=1}^{2}\big(\sum_{k=1}^{J_{0}} f_{i}^{k} \big) + \sum_{i=1}^{2} \alpha_{i}^{J_{0}} \chi_{B(y_{i},2^{J_{0}}r)}.\\
\end{align*}
Now for $k=1,\cdots,J_{0}$, let 
\[
a_{i}^{k} = \frac{f_{i}^{k}}{\|f_{i}^{k}\|_{L^{\infty}}} |B(y_{i},2^{k}r)|^{-1/p}.
\] 
Then, $\text{supp } a_{i}^{k} \subseteq B(y_{i}, 2^{k}r), \|a_{i}^{k}\|_{\infty} = |B(y_{i}, 2^{k}r)|^{-1/p}$, and,
\begin{align*}
\int a_{i}^{k}(x) dx&= \frac{1}{\|f_{i}\|_{L^{\infty}}} |B(y_{i}, 2^{k}r)|^{-1/p} \int f_{i}^{k}(x) dx\\
&= \frac{|B(y_{i}, 2^{k}r)|^{-1/p}}{\|f_{i}\|_{L^{\infty}}} |B(y_{i}, 2^{k}r)|^{-1/p} \int \alpha_{i}^{k-1} \chi_{B(y_{i},2^{k-1}r)} - \alpha_{i}^{k} \chi_{B(y_{i},2^{k}r)} dx\\
&= \frac{|B(y_{i}, 2^{k}r)|^{-1/p}}{\|f_{i}\|_{L^{\infty}}}  \Big( \int_{B(y_{i},2^{k-1}r)} \frac{|B(y_{i},r)|}{|B(y_{i},2^{k-1}r)|} \langle f_{i} \rangle_{B(y_{i},r)} - \int_{B(y_{i},2^{k}r)} \frac{|B(y_{i},r)|}{|B(y_{i},2^{k}r)|} \langle f_{i} \rangle_{B(y_{i},r)}\Big)\\
&= 0.\\
\end{align*}
Thus, $a_{i}^{k}$ is a $p$-atom, and
\[
f=\sum_{i=1}^{2}\big( \sum_{k=1}^{J_{0}} \|f_{i}^{k}\|_{L^{\infty}} |B(y_{i},2^{k}r)|^{1/p} a_{i}^{k} \big) + \sum_{i=1}^{2} \alpha_{i}^{J_{0}} \chi_{B(y_{i},2^{J_{0}}r)}.
\]
It remains to estimate
\[
\sum_{i=1}^{2} \alpha_{i}^{J_{0}} \chi_{B(y_{i},2^{J_{0}}r)}.
\]
To do that, let
\[
\alpha^{J_{0}} = \frac{|B(y_{1},r)|}{|B(\frac{y_{1}+y_{2}}{2}, 2^{J_{0}}r)|} \langle f_{1} \rangle_{B(y_{1},r)}.
\]
Notice that, since $f= f_{1}+f_{2}$ and $\int f(x) dx =0$, we also have
\[
\alpha^{J_{0}} = - \frac{|B(y_{2},r)|}{|B(\frac{y_{1}+y_{2}}{2},2^{J_{0}}r)|} \langle f_{2} \rangle_{B(y_{2},r)}.
\]
Let
\[
f_{i}^{J_{0}+1} = \alpha_{i}^{J_{0}} \chi_{B(y_{i}, 2^{J_{0}}r)} + (-1)^{i} \alpha^{J_{0}} \chi_{B(\frac{y_{1}+y_{2}}{2}, 2^{J_{0}+1}r)},
\]
then,
\begin{align*}
\sum_{i=1}^{2} \alpha_{i}^{J_{0}} \chi_{B(y_{i},2^{J_{0}}r)} &= \alpha_{1}^{J_{0}} \chi_{B(y_{i},2^{J_{0}}r)} + \alpha_{2}^{J_{0}} \chi_{B(y_{i}, 2^{J_{0}}r)} -\alpha^{J_{0}} \chi_{B(\frac{y_{1}+y_{2}}{2}, 2^{J_{0}+1})} + \alpha^{J_{0}} \chi_{B(\frac{y_{1}+y_{2}}{2}, 2^{J_{0}+1})}\\
&= f_{1}^{J_{0}+1} + f_{2}^{J_{0}+1}\\
&= \sum_{i=1}^{2} f_{i}^{J_{0}+1}.
\end{align*}
Let
\[
a_{i}^{J_{0}+1} = \frac{f_{i}^{J_{0}+1}}{\|f_{i}^{J_{0}+1}\|_{L^{\infty}}} |B(\frac{y_{1}+y_{2}}{2}, 2^{J_{0}+1}r)|^{-1/p},
\]
then it is easy to see that $a_{i}^{J_{0}+1}$ is a $p$-atom, (the only tricky part would be to show $a_{i}^{J_{0}+1}$ has mean value zero but that is implied by the above remark on $\alpha^{J_{0}}$), and
\[
f= \sum_{i=1}^{2} \sum_{k=1}^{J_{0}+1} \gamma_{i}^{k} a_{i}^{k},
\]
where 
\[
\gamma_{i}^{k}= 
\begin{cases}
    \|f_{i}^{k}\|_{L^{\infty}} |B(y_{i},2^{k}r)|^{1/p}& \text{for } k=1,\cdots,J_{0}\\
    \|f_{i}^{k}\|_{L^{\infty}} |B(\frac{y_{1}+y_{2}}{2}, 2^{J_{0}+1}r)|^{1/p} & \text{for } k=J_{0}+1.
\end{cases}
\]
Note that, for $k=1,\cdots, J_{0}+1$, by doubling condition, we have
\begin{align*}
|\gamma_{i}^{k}| &= \|f_{i}^{k}\|_{L^{\infty}} |B(y_{i},2^{k}r)|^{1/p}\\
&\leq \frac{|B(y_{i},r)|}{|B(y_{i},2^{k-1}r)|} \|f_{i}\|_{L^{\infty}} |B(y_{i},2^{k}r)|^{1/p}\\
&\lesssim C_{i} (2^{(k-1)})^{n} (2^{kn} |B(y_{i},r)|)^{1/p}\\
&\lesssim C_{i} 2^{kn(\frac{1}{p} -1)} |B(y_{i},r)|^{1/p}.\\
\end{align*}
Thus, $f \in H^{p}(\RR^{n})$, and
\begin{align*}
\|f\|_{H^{p}(\RR^{n})}^{p} &\leq \sum_{i=1}^{2} \sum_{k=1}^{J_{0}+1} |\gamma_{i}^{k}|^{p}\\
&\lesssim \sum_{i=1}^{2} \sum_{k=1}^{{J_0}+1} C_{i} 2^{kn(1 -p)} |B(y_{i},r)|\\
&\leq \sum_{i=1}^{2} C_{i} |B(y_{i},r)| (J_{0}+1) 2^{(J_{0}+1) n (1-p)}\\
&\leq (\sum_{i=1}^{2} C_{i} |B(y_{i},r)|)  N^{n(1-p)} \log_{2} N.
\end{align*}
\end{proof}

Next, we will recall the notion of atomic decomposition. Any $H^{p}(\RR^{n})$ function can be decomposed into an infinite sum of $H^{p}(\RR^{n})$-atoms. We refer the reader to [Theorem 2.3.12, \cite{Gr}] for the proof of the theorem.

\begin{theorem}\label{atomdecomp}
Given a distribution $f \in H^{p}(\RR^{n})$, there exists $\{a_{j}\}_{j=1}^{\infty}$, a sequence of $H^{p}(\RR^{n})$-atoms, and $\{\lambda_{j}\}_{j=1}^{\infty}$ such that
\[
f=\sum_{j=1}^{\infty} \lambda_{j}a_{j} \text{ in } H^{p}(\RR^{n}).
\]
Moreover, we have
\[
\| f \|_{H^{p}(\RR^{n})} \approx \inf \big\{ (\sum_{j=1}^{\infty} |\lambda_{j}|^{p})^{1/p} : f=\sum_{j=1}^{\infty} \lambda_{j}a_{j} \big\}.
\]
\end{theorem}
\begin{theorem}
\label{atomic approximation}
Let $T$ be a bilinear Calder\'on-Zygmund operator, with $\frac{n}{n+1} < \epsilon < 1$, that is 2n-homogeneous in the lth component where $1\leq l\leq2$. Then for all $\varepsilon >0$, there exist $N > 0$ and $C>0$ such that for any $H^{p}(\mathbb{R}^{n})$-atom $a(x)$, there exist $g \in L^{q}(\mathbb{R}^{n})$, $h_{1} \in L^{r_{1}}(\mathbb{R}^{n})$ and $h_{2} \in L^{r_{1}}(\mathbb{R}^{n})$, with $\frac{1}{q}+\frac{1}{r_{1}} + \frac{1}{r_{2}} = \frac{1}{p}$, such that
$$
\|a-\Pi_{l}(g,h_{1},h_{2})\|_{H^{p}(\mathbb{R}^{n})} < \varepsilon\\
$$
and
$$
\|g\|_{L^{q}(\mathbb{R}^{n})} \|h_{1}\|_{L^{r_{1}}(\mathbb{R}^{n})} \|h_{2}\|_{L^{r_{2}}(\mathbb{R}^{n})} \leq CN^{2n}
$$

\end{theorem}

\begin{proof}
Let $\varepsilon >0$ be given and $a(x)$ be an $H^{p}(\mathbb{R}^{n})$-atom with supp $a \in B(x_{0},r)$ for some $x_{0}\in \RR^{n}, r>0$. Fix $1\leq l \leq 2$ and choose $N$ sufficiently large such that $\frac{logN}{N^{\epsilon p - n(1-p)}} < \varepsilon^{p}$.
Choose $y_{l} \in \mathbb{R}^{n}$ such that $y_{l,i} - x_{0,i} =\frac{N r}{\sqrt{n}}, i=1,\cdots,n$, where $x_{0,i}, y_{l,i}$ represent the ith coordinate of $x_{0},y_{l}$, respectively.\\
Now choose $y_{\tilde{l}}$ so that $y_{l}$ and $y_{\tilde{l}}$ satisfy the same relationship as $x_{0}$ and $y_{\tilde{l}}$, i.e. $|y_{l,i} - y_{\tilde{l},i}| = \frac{Nr}{\sqrt{n}}$, where $\tilde{l} = 1$ if $l=2$ and $\tilde{l} =2$ otherwise.\\ In this proof, we will take the case $l=2$ and $\tilde{l} = 1$. The other case is similar.
Let \begin{align*}
g(x) &= \chi_{B(y_{2},r)} (x)\\
h_{1}(x) &= \chi_{B(y_{1},r)} (x)\\
h_{2}(x) &= \frac{a(x)}{{T_{2}}^{*}(h_{1},g)(x_{0})}\\
\end{align*}
Then, 
\[
\text{supp } g = B(y_{2}, r), \text{ supp } h_{1} = B(y_{1}, r) \text{ and supp } h_{2} = B(x_{0},r).
\]
Moreover, we have
\[
\|g\|_{L^{q}(\mathbb{R}^n)} = |B(y_{2}, r)|^{1/q} \approx r^{n/q},\\
\|h_{1}\|_{L^{r_{1}}(\mathbb{R}^n)} = |B(y_{1}, r)|^{1/r_{1}} \approx r^{n/r_{1}},\\
\]
and,
\[
\|h_{2}\|_{L^{r_{2}}(\mathbb{R}^{n})} = \frac{\|a\|_{L^{r_{2}}(\mathbb{R}^{n})}}{|T_{2}^{*}(h_{1},g)(x_{0})|} \lesssim \frac{\|a\|_{L^{\infty}} |B(x_{0},r)|^{1/r_{2}}}{CN^{2n}} \lesssim \frac{r^{-n/p} r^{n/r_{2}}}{CN^{2n}},
\]
where we have used the fact that $T_{2}^{*}$ is $2n$-homogeneous, since T is.\\
Thus,
\[
\|g\|_{L^{q}(\mathbb{R}^{n})} \|h_{1}\|_{L^{r_{1}}(\mathbb{R}^{n})} \|h_{2}\|_{L^{r_{2}}(\mathbb{R}^{n})} \lesssim CN^{2n}.
\]
Let \begin{align*}
f(x)&:= a(x) - \pi_{l}(g,h_{1},h_{2})(x)\\
& = a(x) - (h_{2} T_{2}^{*}(h_{1},g)(x) - gT(h_{1},h_{2})(x))\\
&= \underbrace{a(x) \Big( \frac{T_{2}^{*}(h_{1},g)(x_{0}) - T_{2}^{*}(h_{1},g)(x)}{T_{2}^{*}(h_{1},g)(x_{0})}\Big)}_{W_{1}(x)} - \underbrace{g(x) T(h_{1},h_{2})(x)}_{W_{2}(x)},\\
\end{align*}
and so,
\[
|f(x)| \leq |W_{1}(x)| + |W_{2}(x)|.
\]
Our goal is to get the desired bound for each of $W_{1}$ and $W_{2}$. First note that supp $W_{1} \subseteq B(x_{0},r)$ and supp $W_{2} \subseteq B(y_{2},r)$. So, for $x\in B(x_{0},r)$, by using $2n$-homogeneity of $T_{2}^{*}$, we have
\begin{align*}
|W_{1}(x)| &\leq \|a\|_{L^{\infty}}(\mathbb{R}^{n}) C N^{2n} |T_{2}^{*} (h_{1},g)(x_{0}) - T_{2}^{*}(h_{1},g)(x)|\\
&= r^{-n/p} C N^{2n} \Big| \int K(z_{2},z_{1},x_{0}) h_{1}(z_{1})g(z_{2})dz_{2}dz_{1} - \int K(z_{2},z_{1},x) h_{1}(z_{1})g(z_{2})dz_{2}dz_{1} \Big|\\
&\leq C r^{-n/p}N^{2n} \int_{B(y_{1},r)\times B(y_{2},r)} |K(z_{2},z_{1},x_{0})-K(z_{2},z_{1},x)| dz_{2}dz_{1}\\
&\leq C r^{-n/p}N^{2n} \int_{B(y_{1},r)\times B(y_{2},r)} \frac{|x_{0}-x|^{\epsilon}}{(|z_{2}-z_{1}|+|z_{2}-x_{0}|)^{2n+\epsilon}} dz_{2}dz_{1}\\
&\leq C r^{-n/p}N^{2n} \frac{r^{\varepsilon}}{{(N r)}^{2n+\epsilon}} r^{n} r^{n}\\
&\leq C \frac{1}{r^{n/p}N^{\epsilon}}
\end{align*}\\
Now, for $x\in B(y_{2},r)$, we have
\begin{align*}
|W_{2}(x)| &= |T(h_{1},h_{2})(x)|\\
&= \Big| \int_{\mathbb{R}^{2n}} K(x,y_{1},y_{2}) h_{1}(y_{1}) h_{2}(y_{2}) dy_{1}dy_{2} \Big|\\
&= \Big| \int_{B(y_{1},r)\times B(x_{0},r)} K(x,y_{1},y_{2}) \frac{a(x)}{T_{2}^{*}(h_{1},g)(x_{0})} dy_{1}dy_{2} \Big|\\
&= \Big| \int_{B(y_{1},r)\times B(x_{0},r)} \Big( K(x,y_{1},y_{2})-K(x_{0},y_{1},y_{2})\Big) \frac{a(y_{2})}{T_{2}^{*}(h_{1},g)(x_{0})} dy_{1}dy_{2} \Big|\\
&\leq \int_{B(y_{1},r)\times B(x_{0},r)} \Big| K(x,y_{1},y_{2})-K(x_{0},y_{1},y_{2})\Big| \frac{|a(y_{2})|}{|T_{2}^{*}(h_{1},g)(x_{0})|} dy_{1}dy_{2} \\
&\lesssim \|a\|_{L^{\infty}(\mathbb{R}^{n})} CN^{2n} \int_{B(y_{1},r)\times B(x_{0},r)} \frac{|x-x_{0}|^{\epsilon}}{(|y_{2}-x_{0}|+|y_{1}-x_{0}|)^{2n+\epsilon}} dy_{1}dy_{2}\\
&\lesssim r^{-n/p} C N^{2n} \frac{r^{\epsilon}}{(N r)^{2n+\epsilon}} r^{n} r^{n}\\
&\leq C \frac{1}{N^{\epsilon} r^{n/p}}.
\end{align*}
Therefore,
\[
|f(x)| \leq |W_{1}(x)|+|W_{2}(x)| \lesssim C \frac{1}{N^{\epsilon} r^{n/p}} (\chi_{B({x_{0},r})}+\chi_{B(y_{2},r)}).\\
\]
By previous lemma,
\[
\|f\|_{H^{p}(\mathbb{R}^{n})}^{p} \leq C \frac{log N}{N^{\epsilon p - n(1-p)}} <\varepsilon.
\]
\end{proof}
\bigskip

\section{Proof of Main Results}
For simplicity of notations, we will state and prove Theorems \ref{Thm1Multi} and \ref{Thm2Multi} in the bilinear case.

\begin{theorem}\label{Thm1Bi}
Let $T$ be a bilinear Calder\'on-Zygmund operator that is $2n$-homogeneous in the lth component, with $\frac{n}{n+1} < \epsilon < 1$, where $1\leq l \leq 2$. Then, for every $f \in H^{p}(\RR^{n})$, there exist sequences $ \{ \lambda_{j}^{k} \} \subseteq l^{p}, \{ g_{j}^{k}\} \subseteq L^{q}(\RR^{n}), \{h_{j,1}^{k}\} \subseteq L^{r_{1}}(\RR^{n}), \{h_{j,2}^{k}\} \subseteq L^{r_{2}}(\RR^{n})$, with $\frac{1}{q} + \frac{1}{r_{1}} + \frac{1}{r_{2}} = \frac{1}{p}$, such that
\begin{align}\label{factorization3}
f=\sum_{k=1}^{\infty} \sum_{j=1}^{\infty} \lambda_{j}^{k} \Pi_{l}(g_{j}^{k}, h_{j,1}^{k}, h_{j,2}^{k}) \text{ in } H^{p}(\RR^{n}),
\end{align}
where $$\Pi_{1}(g_{j}^{k}, h_{j,1}^{k}, h_{j,2}^{k}) = h_{j,1}^{k} T_{1}^{*} (g_{j}^{k},h_{j,2}^{k}) - g_{j}^{k} T(h_{j,1}^{k}, h_{j,2}^{k})$$
and $$\Pi_{2}(g_{j}^{k}, h_{j,1}^{k}, h_{j,2}^{k}) = h_{j,2}^{k} T_{2}^{*} (h_{j,1}^{k}, g_{j}^{k}) - g_{j}^{k} T(h_{j,1}^{k}, h_{j,2}^{k}).$$\\
Moreover, we have
\[
\|f\|_{H^{p}(\RR^{n})} \approx \inf \big( \{ \sum_{k=1}^{\infty} \sum_{j=1}^{\infty} |\lambda_{j}^{k}|^{p} \|g_{j}^{k}\|_{L^{q}(\RR^{n})} \|h_{j,1}^{k}\|_{L^{r_{1}}(\RR^{n})} \|h_{j,2}^{k}\|_{L^{r_{2}}(\RR^{n})}\} \big)^{1/p}, \text{ such that } f \text{ satisfies }  \eqref{factorization}.
\]
\end{theorem}

\begin{theorem}\label{Thm2Bi}
Let $T$ be a bilinear Calder\'on-Zygmund operator, that is $2n$-homogeneous in the lth component, with $\frac{n}{n+1} < \epsilon < 1$, $b \in L^{1}_{loc}(\RR^{n})$ and $\alpha>0$ such that $\frac{\alpha}{n} = \frac{1}{p} - 1$. Suppose that $b \in Lip_{\alpha}(\RR^{n})$, then
\[
\|b\|_{Lip_{\alpha}(\RR^{n})} \approx \|[b,T]_{l}\|_{L^{r_{1}}(\RR^{n})\times L^{r_{2}}(\RR^{n}) \rightarrow L^{q'}(\RR^{n})},
\]
where $\frac{1}{q} + \frac{1}{r_{1}} + \frac{1}{r_{2}} = \frac{1}{p}$ and $q'$ is the dual exponent of $q$.

\end{theorem}

To prove Theorem \ref{Thm1Bi}, we will follow Uchiyama's algorithm. Given a function in $H^{p}$, one can use atomic decomposition to decompose this function into atoms. On the other hand, given an atom, one can prove that it can indeed be approximated by the multiplication operator $\Pi_{l}$ in $H^{p}$. Using this result and atomic decomposition, one can decompose $f$ into terms and use an iterative argument to get the desired form \eqref{factorization}.
\begin{proof}[Proof of Theorem \ref{Thm1Bi}]
Let $f\in H^{p}(\RR^{n})$. By Theorem \ref{atomdecomp}, there exist some $\{\lambda_{j}^{1}\} \in l^{p}$, and a sequence $\{ a_{j}^{1} \}$ of $p$-atoms and some constant $C$ such that 
\[
f = \sum_{j=1}^{\infty} \lambda_{j}^{1} a_{j}^{1} \text{ with } (\sum_{j=1}^{\infty} |\lambda_{j}^{1}|^{p})^{1/p} \leq C \|f\|_{H^{p}(\RR^{n})}.
\]
Now let $\varepsilon >0$ so that $ \varepsilon C <1$. By Theorem \ref{atomic approximation}, there exist $\{ g_{j}^{1} \} \subseteq L^{q}(\RR^{n}), \{h_{j,1}^{1}\} \subseteq L^{r_{1}}(\RR^{n}), \{h_{j,2}^{1}\} \subseteq L^{r_{2}}(\RR^{n})$ with 
\[
\|g_{j}^{1}\|_{L^{q}(\RR^{n})} \|h_{j,1}^{1} \|_{L^{r_{1}}(\RR^{n})} \|h_{j,2}^{1} \|_{L^{r_{2}}(\RR^{n})} < C^{\prime} N^{2n},
\]
such that
\[
\| a_{j}^{1} - \Pi_{l}(g_{j}^{1}, h_{j,1}^{1}, h_{j,2}^{1})\|_{H^{p}(\RR^{n})} < \varepsilon \text{ for all } j.
\]
Now, 
\begin{align*}
f &=\sum_{j=1}^{\infty} \lambda_{j}^{1}a_{j}^{1}\\
  &=\sum_{j=1}^{\infty} \lambda_{j}^{1} \Pi_{l}(g_{j}^{1},h_{j,1}^{1}, h_{j,2}^{1}) +\sum_{j=1}^{\infty} \lambda_{j}^{1}\big(a_{j}^{1} - \Pi_{l}(g_{j}^{1}, h_{j,1}^{1},h_{j,2}^{1}) \big)\\
  &=: M_{1} + E_{1}.
\end{align*}
Notice that
\[
\|E_{1}\|_{H^{p}(\RR^{n})} \leq \Big(\sum_{j=1}^{\infty}|\lambda_{j}^{1}|^{p} \|a_{j}^{1} - \Pi_{l}(g_{j}^{1},h_{j,1}^{1},h_{j,2}^{1})\|_{H^{p}(\RR^{n})}^{p} \Big)^{1/p} \leq C \|f\|_{H^{p}(\RR^{n})} < \infty,
\]
so that $E_{1} \in H^{p}(\RR^{n})$ and so Theorem \ref{atomdecomp} implies there exist $\{ \lambda_{j}^{2} \} \in l^{p}$ and $\{a_{j}^{2}\}$, a sequence of $p$-atoms such that
\[
E_{1} = \sum_{j=1}^{\infty} \lambda_{j}^{2} a_{j}^{2} \text{ with } \| \lambda_{j}^{2}\|_{l^{p}} = (\sum_{j=1}^{2} | \lambda_{j}^{2} |^{p})^{1/p} \leq C \|E_{1} \|_{H^{p}(\RR^{n})} \leq C^{2} \varepsilon \|f\|_{H^{p}(\RR^{n})}, 
\]
and now by Theorem \ref{atomic approximation}, for each $p$-atom $a_{j}^{2}$, there exist $g_{j}^{2} \in L^{q}(\RR^{n})$, $h_{j,1}^{2} \in L^{r_{1}}(\RR^{n})$ and $h_{j,2}^{2} \in L^{r_{2}(\RR^{n})}$ with $\|g_{j}^{2}\|_{L^{q}(\RR^{n})} \|h_{j,1}^{2}\|_{L^{r_{1}}(\RR^{n})} \|h_{j,2}^{2}\|_{L^{r_{2}}(\RR^{n})} < C^{\prime} N^{2n}$ such that 
\[
\|a_{j}^{2} - \Pi_{l}(g_{j}^{2}, h_{j,1}^{2}, h_{j,2}^{2})\|_{H^{p}(\RR^{n})} < \varepsilon.
\]
We apply the same iterative argument used on $f$ but this time to $E_{1}$, to get
\begin{align*}
E_{1} &= \sum_{j=1}^{\infty} \lambda_{j}^{2} a_{j}^{2}\\
&= \sum_{j=1}^{\infty} \lambda_{j}^{2} \Pi_{l}(g_{j}^{2}, h_{j,1}^{2}, h_{j,2}^{2}) + \sum_{j=1}^{\infty} \lambda_{j}^{2} (a_{j}^{2} - \Pi_{l}(g_{j}^{2}, h_{j,1}^{2}, h_{j,2}^{2})\\
&=: M_{2} + E_{2},
\end{align*}
which implies,
\[
f=M_{1} + M_{2} + E_{2} = \sum_{k=1}^{2} M_{k} + E_{2}.
\]
Similar to what we did before, we get
\[
\|E_{2}\|_{H^{p}(\RR^{n})} \leq \Big( \sum_{j=1}^{\infty} | \lambda_{j}^{2}|^{p} \|a_{j}^{2} - \Pi_{l}(g_{j}^{2}, h_{j,1}^{2}, h_{j,2}^{2})\|^{p}_{H^{p}(\RR^{n})} \Big)^{1/p} \leq C^{2} \varepsilon ^{2} \|f\|_{H^{p}(\RR^{n})},
\]
and so $E_{2} \in H^{p}(\RR^{n})$.
We keep repeating the same iteration process to get that, for $M \in \mathbb{N}$, $f$ can be represented as
\[
f = \sum_{k=1}^{M} \sum_{j=1}^{\infty} \lambda_{j}^{k} \Pi_{l}(g_{j}^{k}, h_{j,1}^{k}, h_{j,2}^{k}) + E_{M},
\]
where, for $j, k$, $\lambda_{j}^{k} \in l^{p}$, $g_{j}^{k} \in L^{q}(\RR^{n}), h_{j,1}^{k} \in L^{r_{1}}(\RR^{n}), h_{j,2}^{k} \in L^{r_{2}}(\RR^{n})$, with
\begin{align}\label{bound}
\|g_{j}^{k}\|_{L^{q}(\RR^{n})} \|h_{j,1}^{k}\|_{L^{r_{1}}(\RR^{n})} \|h_{j,2}^{k}\|_{L^{r_{2}}(\RR^{n})} \lesssim C N^{2n},
\end{align}
and
\[
E_{M} = \sum_{k=1}^{\infty} \sum_{j=1}^{\infty} \lambda_{j}^{k} (a_{j}^{k} - \Pi_{l}(g_{j}^{k}, h_{j,1}^{k}, h_{j,2}^{k})),
\]
with
\[
(\sum_{j=1}^{\infty} |\lambda_{j}^{k}|^{p})^{1/p} \leq C^{k} \varepsilon^{k-1} \|f\|_{H^{p}(\RR^{n})} \text{ and } \|a_{j}^{2} - \Pi_{l}(g_{j}^{k}, h_{j,1}^{k}, h_{j,2}^{k})\|_{H^{p}(\RR^{n})} < \varepsilon.
\]
This implies $E_{M} \in H^{p}(\RR^{n})$ with $\|E_{M}\|_{H^{p}(\RR^{n})} \leq (\varepsilon C)^{M} \|f\|_{H^{p}(\RR^{n})}$.
Letting $M \rightarrow \infty$, we get that $\|E_{M}\|_{H^{p}(\RR^{n})} \rightarrow 0$, with
\[
f = \sum_{k=1}^{\infty} \sum_{j=1}^{\infty} \lambda_{j}^{k} \Pi_{l}(g_{j}^{k}, h_{j,1}^{k}, h_{j,2}^{k}).
\]
Moreover, since $\varepsilon C <1$, we have
\begin{align*}
\big( \sum_{k=1}^{\infty} \sum_{j=1}^{\infty} | \lambda_{j}^{k}|^{p} \big)^{1/p} &\leq \big( \sum_{k=1}^{\infty} (\varepsilon^{k-1} C^{k} \|f\|_{H^{p}(\RR^{n})})^{p} \big)^{1/p}\\
&\lesssim \|f\|_{H^{p}(\RR^{n})}.
\end{align*}
This, together with \eqref{bound} give us
\[
\|f\|_{H^{p}(\RR^{n})} \gtrsim \inf\{\sum_{k=1}^{\infty} \sum_{j=1}^{\infty} |\lambda_{j}^{k}|^{p} \|g_{j}^{k}\|_{L^{q}(\RR^{n})} \|h_{j,1}^{k}\|_{L^{r_{1}}(\RR^{n})} \|h_{j,2}^{2}\|_{L^{r_{2}}(\RR^{n})}: f \text{ satisfies } \eqref{factorization3}\}.
\]
On the other hand, by a result of \cite{LMW}, we have that, for any $g \in L^{q}(\RR^{n}), h_{1} \in L^{r_{1}}(\RR^{n}), h_{2} \in L^{r_{2}}(\RR^{n})$,
\[
\| \Pi_{l} (g, h_{1}, h_{2})\|_{H^{p}(\RR^{n})} \lesssim \|g\|_{L^{q}(\RR^{n})} \|h_{1}\|_{L^{r_{1}}(\RR^{n})} \|h_{2}\|_{L^{r_{2}}(\RR^{n})}.
\]
Then, for any $f \in H^{p}(\RR^{n})$ having the representation \eqref{factorization3}, we have that
\begin{align*}
\|f\|_{H^{p}(\RR^{n})} &= \| \sum_{k=1}^{\infty} \sum_{j=1}^{\infty} \lambda_{j}^{k} \Pi_{l} (g_{j}^{k}, h_{j,1}^{k}, h_{j,2}^{k}) \|_{H^{p}(\RR^{n})}^{p}\\
&\leq \sum_{k=1}^{\infty} \sum_{j=1}^{\infty} |\lambda_{j}^{k}|^{p} \| \Pi_{l}(g_{j}^{k}, h_{j,1}^{k}, h_{j,2}^{k})\|_{H^{p}(\RR^{n})}\\
&\lesssim \sum_{k=1}^{\infty} \sum_{j=1}^{\infty} | \lambda_{j}^{k}|^{p} \|g_{j}^{k}\|_{L^{q}(\RR^{n})}^{p} \|h_{j,1}^{k}\|_{L^{r_{1}}(\RR^{n})} \|h_{j,2}^{k}\|_{L^{r_{2}}(\RR^{n})}.
\end{align*}
This implies,
\[
\| f \|_{H^{p}(\RR^{n})}^{p} \lesssim \inf \{ \sum_{k=1}^{\infty} \sum_{j=1}^{\infty} | \lambda_{j}^{k}|^{p} \|g_{j}^{k}\|_{L^{q}(\RR^{n})}^{p} \|h_{j,1}^{k}\|_{L^{r_{1}}(\RR^{n})}^{p} \|h_{j,2}^{k}\|_{L^{r_{2}}(\RR^{n})} : f \text{ satisfies } \eqref{factorization3} \}.
\]
This finishes our proof.
\end{proof}

\begin{proof}[Proof of Theorem \ref{Thm2Bi}]
Let $b \in Lip_{\alpha}(\mathbb{R}^{n})$. A more general case of the first inequality was proved in [\cite{LMW}, Theorem 1.1]. For the second part, suppose $f\in H^{p}(\RR^{n})$ such that $\frac{1}{p} = \frac{1}{r_{1}} + \frac{1}{r_{2}} + \frac{1}{q}$, and that $[b,T]_{l}$ maps $L^{r_{1}}(\RR^{n}) \times L^{r_{2}}(\RR^{n}) \text{ to } L^{q^{\prime}}(\RR^{n})$, where $q'$ is the dual exponent of $q$, then by Theorem \ref{Thm1Bi}, there exist sequences $ \{ \lambda_{j}^{k} \} \subseteq l^{p}, \{ g_{j}^{k}\} \subseteq L^{q}(\RR^{n}), \{h_{j,1}^{k}\} \subseteq L^{r_{1}}(\RR^{n}), \{h_{j,2}^{k}\} \subseteq L^{r_{2}}(\RR^{n})$ such that
\begin{align}\label{factorization2}
f=\sum_{k=1}^{\infty} \sum_{j=1}^{\infty} \lambda_{j}^{k} \Pi_{l}(g_{j}^{k}, h_{j,1}^{k}, h_{j,2}^{k}) \text{ in } H^{p}(\RR^{n}),
\end{align}
and so,
\begin{align*}
\langle b, f \rangle_{L^{2}(\RR^{n})} &= \langle b ,  \sum_{k=1}^{\infty} \sum_{j=1}^{\infty} \lambda_{j}^{k} \Pi_{l}(g_{j}^{k}, h_{j,1}^{k}, h_{j,2}^{k}) \rangle_{L^{2}(\RR^{n})}\\
&= \sum_{k=1}^{\infty} \sum_{j=1}^{\infty} \lambda_{j}^{k} \langle b, \Pi_{l}(g_{j}^{k}, h_{j,1}^{k}, h_{j,2}^{k}) \rangle_{L^{2}(\RR^{n})}\\
&= \sum_{k=1}^{\infty} \sum_{j=1}^{\infty} \lambda_{j}^{k} \langle g_{j}^{k}, [b,T]_{l}(h_{j,1}^{k}, h_{j,2}^{k}) \rangle_{L^{2}(\RR^{n})}\\.
\end{align*}
This implies, by Holder's inequality,
\begin{align*}
|\langle b, f \rangle| &\leq \sum_{k=1}^{\infty} \sum_{j=1}^{\infty} |\lambda_{j}^{k}| \|g_{j}^{k}\|_{L^{q}(\RR^{n})} \|[b,T]_{l}(h_{j,1}^{k}, h_{j,2}^{k})\|_{L^{q'}(\RR^{n})}\\
&\leq \sum_{k=1}^{\infty} \sum_{j=1}^{\infty} |\lambda_{j}^{k}| \|g_{j}^{k}\|_{L^{q}(\RR^{n})} \|[b,T]_{l}\|_{L^{r_{1}}(\RR^{n})\times L^{r_{2}}(\RR^{n}) \rightarrow L^{q'}(\RR^{n})} \|h_{j,1}^{k}\|_{L^{r_{1}}(\RR^{n})} \|h_{j,2}^{k}\|_{L^{r_{2}}(\RR^{n})}\\
&\lesssim \|f\|_{H^{p}(\RR^{n})} \|[b,T]_{l}\|_{L^{r_{1}}(\RR^{n})\times L^{r_{2}}(\RR^{n}) \rightarrow L^{q'}(\RR^{n})}.\\
\end{align*}
By duality between $H^{p}(\RR^{n})$ and $Lip_{\alpha}(\RR^{n})$, we have that,
\begin{align*}
\|b\|_{Lip_{\alpha}(\RR^{n})} &\approx \sup_{\|f\|_{H^{p}(\RR^{n})} \leq 1} |\langle b, f \rangle_{L^{2}(\RR^{n})}|\\
&\lesssim \|[b,T]_{l}\|_{L^{r_{1}}(\RR^{n})\times L^{r_{2}}(\RR^{n}) \rightarrow L^{q'}(\RR^{n})}.\\
\end{align*}
\end{proof}

\vspace{0.3cm}

Marie-Jose S. Kuffner, Department of Mathematics, Washington University -- St. Louis, St. Louis, MO 63130-4899 USA

\smallskip

{\it E-mail}: \texttt{mariejose@wustl.edu}
\end{document}